\documentclass[11pt,letterpaper]{amsart}

\usepackage{amsfonts}
\usepackage{amssymb}
\usepackage{graphicx}
\usepackage{amsmath,mathtools}
\usepackage{latexsym}
\usepackage{amscd}
\usepackage{xypic}
\usepackage{mathrsfs}
\usepackage{enumitem} 
\usepackage{braket}
\usepackage{esint}
\usepackage{amsthm}
\usepackage{accents}
\usepackage{microtype}
\usepackage{tikz}
\usetikzlibrary{calc}

\usepackage{hyperref}

\setlist{itemsep=3pt}

\allowdisplaybreaks

\newtheorem{prop}{Proposition}
\newtheorem{theo}[prop]{Theorem}
\newtheorem{lemm}[prop]{Lemma}
\newtheorem{coro}[prop]{Corollary}

\newtheorem*{claim}{Claim}

\theoremstyle{definition}

\newtheorem{rema}[prop]{Remark}
\newtheorem{defi}[prop]{Definition}

\numberwithin{prop}{section}

\newtheorem{conj}{Conjecture}

\newcommand{\bF}{\mathbf{F}}

\newcommand{\HH}{\mathbb{H}}

\newcommand{\MM}{\mathbb{M}}
\newcommand{\NN}{\mathbb{N}}

\newcommand{\RR}{\mathbb{R}}
\renewcommand{\SS}{\mathbb{S}}

\newcommand{\ZZ}{\mathbb{Z}}

\newcommand{\cC}{\mathcal C}

\newcommand{\cF}{\mathcal F}

\renewcommand{\cH}{\mathcal H}
\newcommand{\cI}{\mathcal I}

\newcommand{\cP}{\mathcal P}
\newcommand{\cQ}{\mathcal Q}

\newcommand{\cT}{\mathcal T}

\newcommand{\cZ}{\mathcal Z}

\DeclareMathOperator{\Sym}{Sym}

\DeclareMathOperator{\supp}{supp}

\DeclareMathOperator{\area}{area}

\let\oldmarginpar\marginpar
\renewcommand\marginpar[1]{\-\oldmarginpar[\raggedleft\footnotesize #1]%
{\raggedright\footnotesize #1}}

\DeclareMathOperator{\Id}{Id}

\DeclareMathOperator{\diam}{diam}
\DeclareMathOperator{\length}{length}
\DeclareMathOperator{\Div}{div}

\DeclareMathOperator{\dmn}{dmn}
\DeclareMathOperator{\Index}{index}

\newcommand{\eps}{\varepsilon}

\usepackage{graphicx}

\allowdisplaybreaks

\title{The p-widths of a polygon}
\author{Otis Chodosh}
\author{Sithipont Cholsaipant}
\address{Department of Mathematics, Bldg.\ 380, Stanford University, Stanford, CA 94305, USA}
\email{ochodosh@stanford.edu}
\email{pschol@stanford.edu}

\begin{document}

\begin{abstract}
The $p$-widths are a nonlinear analogue of the spectrum of the Laplacian. We prove that each $p$-width of a polygon in $\mathbb{R}^2$ is achieved by a union of billiard trajectories. We also compute the $p$-widths of the equilateral triangle for $p=1,\dots,4$ and square for $p=1,\dots,3$. 
\end{abstract}
\maketitle

\section{Introduction} 
Consider $\Omega \subset \RR^2$ a compact domain with Lipschitz boundary $\partial\Omega$. (Later in this article we will mostly be interested in polygonal domains.) The $p$-widths of $\Omega$ (denoted $\omega_p(\Omega)$ for $p=1,2,\dots$) are a natural geometric invariant of $\Omega$ (analogous to the spectrum of the Laplacian with Neumann boundary conditions)  introduced by Gromov \cite{Gromov:waist,Gromov:expanders,Gromov:spectra} and studied further by Guth \cite{Guth}. The $p$-widths (of a Riemannian manifold) have received considerable attention due to their link with minimal surfaces (analogous to the eigenfunctions of the Laplacian). See, for example, \cite{Almgren:htpy, Pitts, MarquesNeves:posRic,LiokumovichMarquesNeves,IrieMarquesNeves,Song:full-yau,ChodoshMantoulidis:3d,ZZ:cmc,Zhou:multiplicity-one,MarquesNeves:uper-semi-index}. 

In spite of their importance, only a few $p$-widths are explicitly known (in contrast with the eigenvalues of the Laplacian; cf.\ \cite{mathoverflow:laplace.spectrum}). For regions $\Omega\subset \RR^2$ we have the following results:
\begin{itemize}
\item The unit disk $D$ has $\omega_p(D) \leq \pi d$ for any $p \leq \frac 12 (d^2+3d)$ (Guth \cite[p.\ 1974]{Guth}).
\item The unit disk has $\omega_1(D) = \omega_2(D) = 2$ and $\omega_3(D) = \omega_4(D) = 4$ (Donato \cite{Donato}).
\item The unit disk has $\omega_5(D) > 4$ (Donato--Montezuma \cite[p.\ 60]{DonatoMontezuma}).
\item A sufficiently round ellipse $E$ with diameters $a_1<a_2$ has $\omega_1(E) = a_1$ and $\omega_2(E) = a_2$ (Donato \cite{Donato}).
\end{itemize}
In fact, the only compact Riemannian manifold whose $p$-widths are completely known are the round $2$-sphere which satisfies $\omega_p(\SS^2) = 2\pi \lfloor \sqrt{p}\rfloor$ as proven by the first-named author and Mantoulidis \cite{ChodoshMantoulidis:2d} (earlier work of Aiex computed $\omega_p(\SS^2)$ for $p=1,\dots,8$ \cite{Aiex:ellipsoids}) and the round $\RR P^2$ which satisfies $\omega_p(\RR P^2) = 2\pi \lfloor \frac 14 (1+\sqrt{1+8p}) \rfloor$ as proven by Marx-Kuo \cite{MK:Rp2}. We refer to \cite{Nurser,RL,Chu:ball,BL:RP3,BL:RPn,BL:lens,ChuLi} for results concerning certain $p$-widths of other manifolds (cf.\ \cite{Chu:database}). 

\subsection{The billiard conjecture} In a closed $2$-dimensional surface, the first-named author and Mantoulidis proved \cite{ChodoshMantoulidis:2d} that any $p$-width is equal to the sum of the lengths of closed geodesics (as opposed to geodesic nets, which could \emph{a priori} occur in the theory). The natural analogue of this result for $\Omega\subset \RR^2$ with smooth boundary is contained in the following conjecture (well-known to experts):
\begin{conj}\label{conj:billiard}
    For $\Omega\subset \RR^2$ a bounded domain with smooth boundary and $p=1,2,\dots$ there's a finite set of billiard trajectories $\gamma_{p,1},\dots,\gamma_{p,N(p)} \subset \Omega$ (possibly with repetition) so that
    \[
    \omega_p(\Omega) = \sum_{j=1}^{N(p)} \length(\gamma_{p,N(p)})
    \]
\end{conj}
Here, we define a \emph{billiard trajectory} $\gamma$ to be a union of oriented line segments $\ell_1,\dots,\ell_K$ so that $\partial \ell_k \subset \partial\Omega$ and we require the following two conditions:
\begin{enumerate}
    \item For $k=1,\dots,K$, the ending point $p$ of $\ell_k$ agrees with the starting point of $\ell_{k+1}$ (cyclically defined) and the unit vector parallel to $\ell_{k+1}$ is the reflection of the unit vector parallel to $\ell_k$ across $T_p\partial\Omega$ (i.e.\ the trajectory exhibits specular reflection at $\partial\Omega$). 
    \item One of the following holds:
    \begin{enumerate}
        \item The condition (1) holds cyclically, i.e. for  $\ell_K$ and $\ell_1$, or else
        \item $\gamma_1$ meets $\partial\Omega$ orthogonally at the starting point of $\gamma_1$ and $\gamma_K$ meets $\partial\Omega$ orthogonally at the ending point of $\gamma_K$. 
    \end{enumerate}
\end{enumerate}

\begin{rema}
    The work of Donato--Montezuma  \cite{DonatoMontezuma} implies that if $\Omega\subset \RR^2$ is a convex region with smooth boundary then $\omega_1(\Omega)$ is achieved by a line segment meeting $\partial\Omega$ orthogonally so this verifies Conjecture \ref{conj:billiard} in this case. 
\end{rema}

In this paper we prove an analogue of Conjecture \ref{conj:billiard} when $\Omega = P$ is the compact region bounded by a convex polygon in $\RR^2$. Write $\partial P$ for the boundary of $P$ and $V$ for the set of vertices of $P$. A \emph{billiard trajectory} $\gamma$ in $P$ is defined to be the union of oriented line segments $\ell_1,\dots,\ell_K$ so that either all endpoints lie in $\partial P\setminus V$ and (1),(2.a) hold above or else (1) holds (and all points $p$ lie in $\partial P\setminus V$) and the following holds:
\begin{enumerate}
\item[(1')] For $k=1,\dots,K$, the ending point $p$ of $\ell_k$ lies in $\partial P\setminus V$ and agrees with the starting point of $\ell_{k+1}$ and the specular reflection condition holds at $p$.
\item[(2')] One of the following holds:
\begin{enumerate}
    \item[(a')] The condition (1') holds cyclically, i.e.\ for $\ell_K$ and $\ell_1$, or else
    \item[(b')] the starting point of $\gamma_1$ lies in $\partial P\setminus V$ and $\gamma_1$ meets $\partial P$ orthogonally at the starting point, or the starting point of $\gamma_1$ is contained in $V$ and a  similar statement holds for the endpoint of $\gamma_K$. 
\end{enumerate}
\end{enumerate}
Loosely stated, a billiard trajectory in a polygon is a sequence of straight lines that reflect off of interior boundary points with equal incidence angle and stop if they hit a vertex.
\begin{rema}
    Observe that in this definition of billiard trajectory, the billiard will stop if it hits a vertex. For example, if $P$ is a convex polygon then any side of $P$ will be a billiard trajectory. We conjecture that for certain polygons, some $p$-width is attained by the length of a single side. However, this will not be the case for all polygons. In particular, in Theorem \ref{prop:triangle-billiards} we prove that billiard trajectories in the equilateral triangle $T$ unfold to straight lines in $\RR^2$. In particular, this allows for $\partial T$ to occur, but not one or two sides of $T$.
\end{rema}

Our first result here is as follows: 
\begin{theo}\label{theo:polygon-billiard}
    For $P\subset \RR^2$ a (compact) convex polygon and $p=1,2,\dots$ there's a finite set of billiard trajectories $\gamma_{p,1},\dots,\gamma_{p,N(p)} \subset P$ (possibly with repetition) so that
    \[
    \omega_p(P) = \sum_{j=1}^{N(p)} \length(\gamma_{p,j})
    \]
\end{theo}
The proof of Theorem \ref{theo:polygon-billiard} follows the approach used by the first-named author and Mantoulidis to prove \cite[Theorem 1.2]{ChodoshMantoulidis:2d} along with a reflection argument reducing the problem to a question on the interior. 

\begin{rema}
The reflection argument avoids the issue that a boundary version of the Wang--Wei curvature estimates for the Allen--Cahn equation \cite{WangWei,WangWei2,ChodoshMantoulidis:3d} not known (such estimates would likely resolve Conjecture \ref{conj:billiard}). 
\end{rema}

\begin{rema}
    Theorem \ref{theo:polygon-billiard} is easily seen to extend to $(M,g)$ a Riemannian metric on a compact surface so that $\partial M$ is made up of piecewise totally geodesic curves with convex corners. In this case, in the definition of billiards we should replace ``line segments'' by ``geodesic segments'' and also allow for a closed geodesic in $M$. 
\end{rema}

\subsection{The first width of a convex polygon}

For a compact (convex) set $K \subset \RR^2$ we define the geometric width 
\[
W(K) : = \inf_{|v|=1} \left( \max_{x\in K} x\cdot v - \min_{x\in K} x\cdot v \right).
\]
In other words, $W(K)$ is the width of the smallest slab that contains $K$. 
\begin{theo}\label{theo:poly-w1}
    For a convex polygon $P\subset \RR^2$ we have $\omega_1(P) = W(P)$. 
\end{theo}
Note that the upper bound $\omega_1(P) \leq W(P)$ follows by sweeping out by parallel lines in the direction of an optimal choice of $v$. 

\begin{rema}
    If we consider the $2$-sweepout by all lines then we would get $\omega_2(P)\leq \diam(P)$ (equivalently, replace ``$\inf$'' by ``$\sup$'' in $W(\cdot)$). However, this inequality does not need to be sharp, since in Theorem \ref{theo:low-widths-triangle} we show that for $T$ an equilateral triangle inscribed in the unit circle, we have $\omega_1(T)=\omega_2(T) = \frac 32$ while $\diam(T) =\sqrt{3}$.
\end{rema}

\subsection{Low widths of certain regular polygons}

Let $T$ denote the closed region enclosed by the equilateral triangle that's inscribed in the unit circle and $S$ the closed interior of the square of sidelength $\sqrt{2}$ inscribed in the unit circle we can compute the following $p$-widths.

\begin{theo}\label{theo:low-widths-triangle}
We have $\omega_1(T) = \omega_2(T)=\frac{3}{2}$, $\omega_3(T)=\frac{3\sqrt{3}}{2}$, and $\omega_4(T)=3$.
\end{theo}
In terms of the corresponding billiard, that $\frac 32$ is the length of the perpendicular bisector, $\frac{3\sqrt{3}}{2}$ is the length of the smaller equilateral triangle whose vertices are the midpoints of the edges of $T$ and $3$ is the combined length of two perpendicular bisectors. 

\begin{theo}\label{theo:low-widths-square}
We have $\omega_1(S) = \sqrt{2}$,  $\omega_2(S)=2$, and $\omega_3(S)=2\sqrt{2}$.
\end{theo}
Note that $\sqrt{2}$ is a vertical/horizontal segment and $2$ is a diagonal. 

\subsection{Organization}
The preliminaries are contained in Section \ref{sec:prelim}. The billiard theorem for arbitrary convex polygons is proven in Section \ref{sec:billiard} and is improved in the special case of an equilateral triangle in Section \ref{sec:triangle-billiards}. The first $p$-width of a convex polygon is computed in Section \ref{sec:w1} and some of the low $p$-widths of an eqilateral triangle (resp.\ square) are computed in Section \ref{sec:triangle} (resp.\ \ref{sec:square}). A brief discussion of Caccioppoli sets and flat chains modulo $2$ is contained in Appendix \ref{app:caccioppoli-sets}. 

\subsection{Acknowledgments}
O.C. was supported by a Terman Fellowship and an NSF grant (DMS-2304432). We are grateful to Michael Ren for showing us the proof of Lemma \ref{lemm:tet-intersection-length} and to Christos Mantoulidis for his interest in this note and for many discussions concerning the $p$-widths.

\section{Preliminaries} \label{sec:prelim}
We give a brief overview of the $p$-widths in the setting at hand. For a more comprehensive overview one may refer to \cite{MarquesNeves:app} and \cite[\S 2]{ChodoshMantoulidis:2d}. Below, $K$ will be a compact set in $\RR^2$ with Lipschitz boundary.

\subsection{Gromov--Guth $p$-widths} We refer to Appendix \ref{app:caccioppoli-sets} for the definitions of Caccioppoli sets and the set of relative flat $1$-chains $\cZ_1(K,\partial K)$. Facts (1) and (2) in Appendix \ref{app:caccioppoli-sets} give the following result  (due to Almgren \cite{Almgren:htpy} and Marques--Neves  \cite[Theorem 2.3]{MarquesNeves:app})
\begin{lemm}\label{lemm:cycles-weak-homotopy-equiv}
The map $\hat \Phi : \RR P^\infty \to \cZ_1(K,\partial K)$ defined by 
\[
[a_0:\dots:a_k:0:\dots] \mapsto \{ (x,y) \in K : a_0  + a_1 x + a_2 x^2 + \dots a_k x^k < 0 \}
\]
is a weak homotopy equivalence. 
\end{lemm}
The Hurewicz theorem gives that $H^*(\cZ_1(K,\partial K);\ZZ_2) = \ZZ_2[\bar \lambda]$. 

For $X$ a finite dimensional cubical subcomplex of some cube $I^m$, a map 
\[
\Phi : X \to \cZ_1(K,\partial K)
\] 
has no concentration of mass if 
\[
\lim_{r\to 0} \sup\{ \MM(\Phi(x)\cap B_r(p)) : x \in K,p \in K\} = 0. 
\]

\begin{defi}
For $p \in \NN$, a $p$-sweepout is a continuous map $\Phi : X \to \cZ_1(K,\partial K)$ where $X$ is a finite dimensional cubical subcomplex of some cube $I^m$ so that $\Phi$ has no concentration of mass and $\Phi^*(\bar \lambda^k) \neq 0 \in H^k(X;\ZZ_2)$. 
\end{defi}
The following well-known lemma will be useful for checking that certain explicit maps are $p$-sweepouts:
\begin{lemm}\label{lemm:check-sweepout}
A map $\Phi : X\to \cZ_1(K,\partial K)$ is a $p$-sweepout if and only if the following holds:
\begin{itemize}
\item There's $\lambda \in H^1(X;\ZZ_2)$ so that $\lambda^p \neq 0 \in H^p(X;\ZZ_2)$ and
\item  for any $\gamma : S^1 \to X$ we have $\lambda(\gamma) \neq 0$ if and only if $\Phi \circ \gamma : S^1 \to \cZ_1(K,\partial K)$ is a $1$-sweepout. 
\end{itemize}
\end{lemm}
Note that the Hurewicz theorem shows that $\tilde \Phi : S^1 \to \cZ_1(K,\partial K)$ is a $1$-sweepout if and only if $[\tilde \Phi] \not = 0 \in \pi_1(\cZ_1(K,\partial K),\emptyset)$, i.e.\ $\tilde\Phi$ does not lift to a loop under the double cover $\cC(K,\partial K) \to \cZ_1(K,\partial K)$.

We write $\cP_p=\cP_p(K)$ for the set of $p$-sweepouts. We can then define the $p$-widths as
\[
\omega_p(K) : = \inf_{\Phi \in \cP_p} \sup\{\MM(\Phi(x)) : x \in \dmn(\Phi)\}. 
\]
It's easy to see that $\omega_1(K)\leq\omega_2(K)\leq\dots$. We also have
\[
\lim_{p\to\infty} \omega_p(K)p^{-\frac 12} = \sqrt{\pi\area(K)}
\]
by \cite{LiokumovichMarquesNeves,ChodoshMantoulidis:2d}. Finally, we recall the Lusternick--Schnirelman inequality.
\begin{lemm}\label{lemm:LS}
    For $K,K_1,\dots,K_J\subset\RR^2$ Lipschitz domains with $K_1,\dots,K_J\subset K$ and the interior of $K_1,\dots,K_J$ pairwise disjoint, if $p_1+\dots+p_J = p$ then
    \[
    \omega_p(K) \geq \sum_{j=1}^J \omega_{p_j}(K_j). 
    \]
\end{lemm}
This is well-known. See \cite[3.1]{LiokumovichMarquesNeves} for the case of Lipschitz domains.

\subsection{Double-well phase transition theory} 
We now describe an alternative method for computing the widths through a regularization of the problem. A smooth function $W:\RR\to\RR$ is a \emph{double well potential} if $W\geq 0$, $W(-t)=W(t)$, $tW'(t)<0$ for $0<|t|<1$, and $W''(\pm 1) > 0$. For $W$ a fixed double-well potential we can define the $\eps$-phase transition energy
\[
E_\eps[u] : = \int_K \frac \eps 2 |\nabla u|^2 + \frac 1 \eps W(u).
\]
A critical point of $E_\eps$ satisfies the Allen--Cahn equation
\[
\eps^2\Delta u = W'(u)
\]
along with the Neumann condition $\partial_\eta u =0$ along $\partial K \setminus V$. 

The second variation of $E_\cdot$ at a critical point is 
\[
\delta^2E_\eps|_u[\varphi] = \int_K \eps |\nabla \varphi|^2 + \frac 1 \eps W''(u) \varphi^2. 
\]
allowing us to define  $\Index_\eps(u;U)$ to be the maximal dimension of $V\subset C^\infty_c(U)$ so that $\delta^2E_\eps|_u[\varphi] < 0$ for $\varphi \in V\setminus\{0\}$ for $U\subset K$. We write $\Index_\eps(u) = \Index_\eps(u,K)$.

We let $\tilde\cP_p$ denote the space of $\ZZ_2$ equivariant maps $h : \tilde X \to H^1(K)\setminus\{0\}$ where $\tilde X \to X$ is a double cover of a $p$-dimensional cubical subcomplex of $I^m$ and $H^1(K)$ is the Sobolev space of functions whose first derivative is in $L^2$. We can define the $\eps$-phase transition widths by 
\[
\omega_{p,\eps}(K) := \inf_{h\in \tilde\cP_p} \max_{x\in\dmn(h)} E_\eps(h(x)).
\]
The following result shows the existence of a critical point achieving the width (at least for $\eps \leq \eps_0(p,K)$).
\begin{theo}[{\cite[Proposition 4.4]{Guaraco:1param}, \cite[Theorem 3.3]{GasparGuaraco1}, \cite[\S 2.4]{Dey:comp}}]\label{theo:allen-cahn-min-max}
If $\omega_{p,\eps}(K) < |K|/\eps$ then there's a critical point $u$ of $E_\eps$ with $E_\eps[u] = \omega_{p,\eps}(K)$, $|u_\eps| < 1$, and $\Index(u_\eps)\leq p$. 
\end{theo} 
\subsection{Comparison between the min-max theories}
Let $h_0 : = \int_{-\infty}^\infty \HH'(t)^2 + W'(\HH(t))$ where $\HH(t)$ is the heteroclinic solution i.e.\ $\HH''(t) = W'(\HH(t))$ with $\HH(t) \to \pm 1$ when $t\to\pm\infty$. 

\begin{prop}[{\cite{Dey:comp} cf.\ \cite[Theorem 6.1]{GasparGuaraco1}}]\label{prop:dey-bdry}
    We have
    \[
    h_0^{-1}\lim_{\eps\to 0}\omega_{p,\eps}(K) = \omega_p(K). 
    \]
\end{prop}
This is stated only for closed manifolds in \cite{Dey:comp} but a careful examination of the proof shows that it's still valid for manifolds with smooth boundary as well as for polygonal regions in $\RR^2$.

\section{Proof of the billiard theorem}\label{sec:billiard}
In this section we prove Theorem \ref{theo:polygon-billiard}. 

Fix $P$ a compact convex polygonal region in $\RR^2$ and fix $p\in \{1,2,\dots\}$. For $0< \eps \ll1$, let $u_\eps$ denote the $p$-parameter min-max critical point of $E_\eps$ guaranteed by Theorem \ref{theo:allen-cahn-min-max}. Later we will choose the ``sine--Gordon'' double well potential but we leave the choice of $W$ arbitrary for now.  
We have that $u_\eps$ solves
\[
\begin{cases}
    \eps^2 \Delta u_\eps = W''(u_\eps) & \textrm{in the interior of $P$}\\
    \partial_\eta u = 0 & \textrm{on $\partial P \setminus V$}
\end{cases}
\]
as well as $E_\eps[u]=\omega_{p,\eps}(K)$ and $\Index_\eps(u_\eps) \leq p$. 
\begin{lemm}\label{lemm:grad-vanishes-corner}
    $u\in C^1(P)$ and $|\nabla u| = 0$ at the vertices $V$.
\end{lemm}
\begin{proof}
Since $P$ is convex, all vertices have angles $<\pi$. Thus, we have that $u_\eps \in C^1(P)$ by e.g.\ \cite{AK:elliptic}. Since $u_\eps$ has Neumann boundary conditions on all edges, we thus have $|\nabla u_\eps|=0$ at a vertex (since the gradient needs to vanish in two linearly independent directions). 
\end{proof}
Using this we can prove non-positivity of the ``discrepancy function'' as in \cite[Lemma 3.5]{HT}:
\begin{coro}
    We have $\frac \eps 2 |\nabla u_\eps|^2 \leq \frac{1}{\eps} W(u_\eps)$ on $P$. 
\end{coro}
\begin{proof}
    By scaling we can assume $\eps=1$ which we now do for simplicity. Let $\xi = \frac 12 |\nabla u|^2 - W(u)$. Note that $\xi \leq 0$ at a vertex by combining Lemma \ref{lemm:grad-vanishes-corner} with $W \geq 0$. Thus, if $\max_P \xi >0$ the maximum is either attained at an interior point or else on $\partial P\setminus V$ (the interior of an edge). In the latter case we can locally reflect across the edge (even reflection) to reduce to the case of an interior (local) maximum. This is a contradiction by applying the maximum principle to \cite[(3)]{Modica} (cf.\ \cite[Lemma 3.5]{HT}). 
\end{proof}

The energy density $\mu_\eps:= \frac{\eps}{2}|\nabla u_\eps|^2 + \frac 1 \eps W(u_\eps) \, dx dy$ defines a Radon measure on $K$ with mass $\omega_{p,\eps}(K)$. We can pass to a subsequence $\eps_j\to 0$ and assume that $\mu_{\eps_j} \rightharpoonup \mu$ for $\mu$ a Radon measure (actually one may prove that this convergence occurs in the varifold sense, cf.\ \cite{HT}). 

\begin{lemm}\label{lemm:mass-bd-vertex}
    For $p \in V$ and $0<r\leq r_0=r_0(p)$ independent of $\eps$ we have $\mu_\eps(B_r(p)) \leq C r$ for $C$ independent of $\eps,r$.
\end{lemm}
\begin{proof}
    Fix $r_0$ small enough so that $V\cap B_{r_0}(p) = \{p\}$. The, as in \cite[Lemma 3.1]{HT} or \cite[\S 4]{LPS} the monotonicity formula for Allen--Cahn equation, along with non-positivity of the discrepancy function gives
    \[
    r^{-1} \mu_\eps(B_r(p)) \leq r_0^{-1}\mu_\eps(B_{r_0}(p))
    \]
    for $0<r<r_0$. Since the boundary of $P$ consists of straight lines and $u_\eps$ satisfies Neumann conditions there, no boundary correction (as in \cite{LPS}) is necesssary. 
\end{proof}
We now fix $W(t) = \frac{1+\cos(\pi t)}{\pi^2}$ so that the integrability results from \cite{LW,ChodoshMantoulidis:2d}  are applicable. 
\begin{prop}\label{prop:local-bounce}
    For any $x \in P\setminus V$, there's $r>0$ so that $V\cap B_r(x) = \emptyset$ and $\mu\lfloor B_r(x) = \sum_{j=1}^J  \cH^1\lfloor \eta_j$ for $\eta_j$ line segments in $P \cap B_r(x)$ with endpoints in $\partial(P \cap B_r(x))$. If $x \in \partial P\setminus V$ then segments $\eta_j$ that are not contained in $\partial P$ but intersect $\partial P$ come in pairs that bounce off of $\partial P$ and a segment $\eta_j \subset \partial P \cap B_r(x)$ must be equal to $\partial P \cap B_r(x)$. 
\end{prop}
Here $\cH^1$ is the $1$-dimensional Hausdorff measure (i.e.\ $(\cH^1\lfloor \ell)(A)$ is the $1$-dimensional Lebesgue measure of $\ell \cap A$). 
\begin{proof}
    For $x$ in the interior of $P$ this follows from the proof of \cite[Theorem 1.2]{ChodoshMantoulidis:2d}. For $x \in \partial P \setminus V$ choose $r>0$ sufficiently small so that $V\cap B_r(x)=\emptyset$. Up to a dilation, translation, and isometry we thus can assume that $u_{\eps_j}$ solves $\eps^2 \Delta u_{\eps_j} = W'(u_\eps)$ on $B^+ = \{(x,y) \in B_1 : y\geq 0\}$ with $\partial_y u_{\eps_j} = 0$ on $\partial_0B^+ = \{(x,0) : |x|<1\}$. Let $\tilde u_{\eps_j}$ denote the even reflection across $\partial_0B^+$. Elliptic estimates give $\tilde u_{\eps_j} \in C^\infty_\textrm{loc}(B)$ for $B=B_1$. 
    \begin{claim}
        $\Index_{\eps_j}(\tilde u_{\eps_j};B) \leq 2p$.
    \end{claim} 
    \begin{proof}
        This follows from standard resutls (cf.\ \cite[(2.5)]{ABCS}). For completeness we give the proof below. 

        Let $\mathring B^+ : = B^+\setminus\partial_0B^+$ be the interior of the half disk. Since $C^\infty_c(\mathring B^+) \subset C^\infty_c(B^+)$ we have that 
        \[
        I_D : = \Index_{\eps_j}(u_{\eps_j};\mathring B_+) \leq \Index_{\eps_j}(u_{\eps_j};B_+) : = I_N \leq p. 
        \]
        By the variational characterization of eigenvalues of Schrodinger operators, we have $I_D$ (resp.\ $I_N$) is equal to the number of negative eigenvalues for the operator $\eps_j^2 \Delta - W''(u_\eps)$ with Dirichlet (resp.\ Neumann) boundary conditions on $\partial_0 B^+$ and Dirichlet boundary conditions on $\partial B^+ \setminus \partial B_0^+$. Let $W \subset C^\infty(B)$ denote the span of all of the odd reflections of eigenfunctions corresponding to $I_D$ and all even reflection of eigenfunctions corresponding to $I_N$. 
        
        Suppose that $V\subset C^\infty(B)$ has $\dim V>2p$. We can find $\varphi \in V$ that's $L^2$-orthogonal to $W$. Write $\varphi = \varphi_E+\varphi_O$ for the even and odd parts (with respect to $y\mapsto -y$). Note that even functions are $L^2$-orthogonal to odd functions. This gives that $\varphi_E$ is $L^2$-orthogonal to the even reflection of all of the eigenfunctions for $I_N$ and similarly for $\varphi_O$ and $I_D$. Thus we find that
        \[
        \delta^2E_{\eps_j}[\varphi] = \delta^2E_{\eps_j}[\varphi_E] + \delta^2E_{\eps_j}[\varphi_O] \geq 0
        \]
        by the variational characterization of eigenvalues. This proves the claim. 
    \end{proof}

    Thus, the proof of \cite[Theorem 1.2]{ChodoshMantoulidis:2d} implies that the doubled energy density $\tilde\mu_{\eps_j}$ (subsequentially) limits to $\sum_{j=1}^J \cH^1\lfloor \tilde \eta_j$ where $\tilde \eta_1,\dots,\tilde \eta_J$ are line segments with endpoints in $\partial B$. The set of line segments (with repetition allowed) must be invariant under reflection $y\mapsto -y$ since $\tilde \mu_{\eps_j}$ is. 

    \begin{claim}
        The boundary segment $\partial_0B^+$ occurs in  the list $\tilde \eta_1,\dots,\tilde \eta_J$ with even multiplicity. 
    \end{claim}
    \begin{proof}
        Choose $q \in \partial_0B^+$ and $r>0$ so that $B_r(q) \subset B$ and if $\tilde\eta_j$ is in the list above and $\tilde\eta_j \neq \partial_0 B^+$ then $B_r(q) \cap \tilde\eta_j = \emptyset$. The restricted functions $\tilde u_{\eps_j}|_{B_r(q)}$ have corresponding energy measures $\tilde\mu_{\eps_j}\lfloor B_r(q)$ converging to $k \cH^1 \lfloor (\partial_0 B^+ \cap B_r(q))$ for some $k \in \NN$ (the multiplicity that $\partial_0B^+$ occurs in the list). By \cite[Theorem 1(2)]{HT} and the fact that $\tilde u$ is an even reflection across $\partial_0B^+$ we have that either $u_{\eps_j}|_{B_r(q)} to 1$ or $-1$ locally uniformly on $B_r(q)\setminus \partial_0B^+$. Thus, by \cite[Theorem 1(4)]{HT}, the multiplicity $k$ is even. 
    \end{proof}
    Observe that for $x \in B^+\setminus \partial_0B^+$ we have that $\mu(B_r(x)) = \tilde\mu(B_r(x))$ for $r<d(x,\partial B_+)$ and for $x \in \partial_0 B^+$ and $r<1$ we have $\mu(B_r(x)\cap B^+) = \frac 12 \tilde \mu(B_r(x))$. Thus, if we list the line segments $\tilde \eta_1,\dots,\tilde\eta_K,\tilde\eta_{K+1},\dots,\tilde\eta_J$ where $\tilde\eta_j = \partial_0B^+$ for $j = K+1,\dots,J$ (say $2M$ times) we see that
    \[
    \mu = \sum_{j=1}^K \cH^1\lfloor( \tilde\eta_j \cap B^+) + M \cH^1\lfloor \partial_0B^+
    \]
    where each $\tilde\eta_j$, $j=1,\dots,K$ has a reflected pair in this list. This completes the proof.
    \end{proof}

We can finally prove Theorem \ref{theo:polygon-billiard} as follows. By Proposition \ref{prop:local-bounce}, for any $s>0$ we see that
\[
\mu\lfloor (P\setminus B_s(V)) = \sum_{j=1}^J \cH^1\lfloor \hat \eta_j
\]
for line segments $\hat \eta_j$ in $P\setminus B_s(V)$ with boundary in $\partial(P\setminus B_s(V))$ and so that if a boundary point of $\hat \eta_j$ is in $\partial P \setminus B_s(V)$ then there's some other $\hat \eta_{j'}$ with a common boundary point so that $\hat \eta_j$ and $\hat \eta_{j'}$ bounce off of $\partial P$ at the common boundary. 

Since billiards of finite length cannot bounce infinitely many times into a convex corner (cf.\ \cite[Lemma 2.3]{Lange:billiard}) we thus see that 
\[
\mu\lfloor (P\setminus V) = \sum_{j=1}^J \cH^1\lfloor \eta_j
\]
for line segments $\eta_j$ in $P\setminus V$ with boundary in $\partial P$ and so that if a boundary point of $\eta_j$ is in $\partial P \setminus V$ then there's some other $\eta_{j'}$ with a common boundary point so that $\eta_j$ and $\eta_{j'}$ bounce off of $\partial P$ at the common boundary. 

Lemma \ref{lemm:mass-bd-vertex} implies that $\mu(V) = 0$, and thus we have
\[
\mu  = \sum_{j=1}^J \cH^1\lfloor \eta_j.
\]
 for $\eta_j$ as above. Grouping the $\eta_j$ into billiard trajectories by following successive bounces (unless the endpoint is in $V$ or the interval meets $\partial P$ orthogonally in which case the billiard terminates), this proves Theorem \ref{theo:polygon-billiard}.

\section{Improving the billiard theorem for the equilateral triangle} \label{sec:triangle-billiards} 

In this section we show that billiards that arise in the $p$-widths of an equilateral triangle $T$ ``bounce'' off of vertices in a precise way. For example, this will imply that a single (or double) side of $T$ cannot arise as a billiard in $\omega_p(T)$ unless the full $\partial T$ does. 

For a vertex $v \in V$ write $\ell_v$ for the line that passes through $v$ and is parallel to the side of $T$ opposite of $v$. We say that a \emph{$T$-billiard trajectory} is a sequence of intervals $\eta_1,\dots,\eta_K$
\begin{enumerate}
\item[(1'')] For $k=1,\dots,K$, the ending point $p$ of $\ell_k$ lies in $\partial T$ and agrees with the starting point of $\ell_{k+1}$. If $p \in \partial T\setminus V$, the specular reflection condition holds at $p$ with respect to the edge that $p$ lies in. If $p \in V$ the specular reflection condition holds at $p$ with respect to $\ell_v$. 
\item[(2'')] One of the following holds:
\begin{enumerate}
    \item[(a'')] The condition (1') holds cyclically, i.e.\ for $\ell_K$ and $\ell_1$, or
    \item[(b'')] The first interval $\eta_1$ meets $\partial T\setminus V$ orthogonally at its starting point or else starts at $v\in V$ and meets $\ell_v$ orthogonally.  Similar consideration holds for the endpoint of $\eta_K$. 
\end{enumerate}
\end{enumerate}

Consider the tesselation of $\RR^2$ by repeated reflections across edges of $T$. Let $G$ denote the group generated by reflections preserving the tesselation and write $\cT$ for the image of $\partial T$ under the tesselation (an infinite set of straight lines). 

\begin{lemm}\label{lemm:triangle-billiard-reflect}
    Suppose that $\eta_1,\dots,\eta_K \subset T$ is a set of line segments with $\partial\eta_k \subset \partial T$. Suppose that the image of the segments under the reflection group $G(\cup_k \eta_k)\subset \RR^2$ is a set of straight lines (with constant multiplicity). Then $\eta_1,\dots,\eta_K$ is a union of $T$-billiards. 
\end{lemm}
\begin{proof}
    Let $p$ be an endpoint of $\eta_1$. If $p \in \partial T\setminus V$ then by considering the reflection of $T$ across the edge containing $p$ we see that either $\eta_1$ meets $\partial T$ orthogonally at $p$ or else there's $\eta_2$ (up to relabeling) so that $\eta_1$ bounces off of $\partial T$ at $p$ to give $\eta_2$. If $p \in V$ then note that the reflection of $T$ across $\ell_v$ is in the tesselation so similar logic applies there. This completes the proof. 
\end{proof}

We now show that the $p$-widths of $T$ are attained by $T$-billiards:
\begin{prop}\label{prop:triangle-billiards}
    Let $T$ be an equilateral triangle. Then for $p=1,2,\dots$ we have
    \[
    \omega_p(P) = \sum_{j=1}^{N(p)} \length(\gamma_{p,j})
    \]
    for $T$-billiard trajectories $\gamma_{p,j}$. 
\end{prop}

\begin{rema}
    One may extend this theorem to a more general class of polygons. The crucial property is that the polygon admits an edge tesselation with an even number of faces meeting any vertex. This guarantees that if a sequence of reflections brings $P$ to $P$ then this reflection will be the identity map, rather than a self-reflection (for example, the regular hexagon does not have this property). The theorem also clearly extends to such polygons in the hyperbolic plane.
\end{rema}

\begin{proof}[Proof of Proposition \ref{prop:triangle-billiards}]
For $p \in \{1,2,\dots\}$ and $0 < \eps \ll 1$, let $u_\eps$ be the solution to the $\omega_{p,\eps}(T)$ min-max problem, so in particular
\[
\begin{cases}
    \eps^2 \Delta u_\eps = W''(u_\eps) & \textrm{in the interior of $T$}\\
    \partial_\eta u = 0 & \textrm{on $\partial T \setminus V$}
\end{cases}
\]
as well as $E_\eps[u]=\omega_{p,\eps}(K,\partial K)$ and $\Index_\eps(u_\eps) \leq p$. Let $\tilde u_\eps$ denote the function\footnote{Note that if we tried to repeat this proof for a hexagon, the even reflection will not be well-defined on $\RR^2$.} on $\RR^2$ generated by (even) reflection across the edges of $T$. As in the proof of Theorem \ref{theo:polygon-billiard} we see that $\tilde u_\eps$ has uniformly bounded energy and index on compact subsets of $\RR^2$ and thus a subsequential limiting energy measure $\mu$ must satisfy $\tilde \mu = \sum_{\ell \in L} \cH^1\lfloor \ell$ for straight lines $\ell$ by \cite[Theorem 1.2]{ChodoshMantoulidis:2d}. The measure $\tilde \mu$ is invariant under the reflections. Arguing as in Theorem \ref{theo:polygon-billiard}, this implies the claim. 
\end{proof}

\begin{lemm}\label{lemm:len-T-Bil}
    Suppose that $\gamma$ is a $T$-billiard. Then 
    \[
    \length(\gamma) \in \{\tfrac32\sqrt{a^2+ab+b^2} : a,b\in\ZZ\}.
    \]
\end{lemm}
\begin{proof}
    It suffices to consider $T$ with vertices $A=(0,0),B=(\sqrt{3},0),C=(\frac{\sqrt{3}}{2},\frac 32)$. We assume that $\eta_1$ starts on $\overline{AB} \setminus B$. Let $\tilde\gamma \subset \RR^2$ denote the \emph{unfolding} of $\gamma$ to $\RR^2$. This means that we start with $\eta_1$ and then reflect $\eta_2$ across the edge of $\partial T$ (or else $\ell_v$) to straighten out the specular reflection. This yields a line segment $\tilde\gamma \subset \RR^2$ with length $=\sum_{k=1}^K |\eta_j|$. 
    
    Assume that case (2.b'') holds for the starting point. Then $\tilde\gamma$ will be a vertical line (since $\ell_A$ is horizontal) and thus we see that the endpoint of $\tilde\gamma$ must lie on a horizontal line in the tesselation by $T$. This gives that $|\tilde\gamma| = m \frac 32$. 

    We now assume that (2,a'') holds. Then the endpoint of $\tilde\gamma$ lies in some reflection of $\overline{AB}$ which is contained in some line $L \in \cT$. If $L$ is not horizontal, we consider $\hat \gamma$ formed by the union of $\tilde\gamma$ with its reflection across $L$. This will again be a straight line segment (by the cyclic condition). If $L$ is horizontal we set $\hat\gamma=\tilde\gamma$. Since the endpoint of $\hat \gamma$ is a reflection of the starting point, we can translate $\hat \gamma$ so its starting point is $A$ and its endpoint is in $G(A)$. Note that $G(A)$ is a lattice in $\RR^2$ generated by $(0,3)$ and $(\frac{3\sqrt{3}}{2},\frac32)$. Thus the length of $\hat\gamma$ is of the form
    \[
    \left\| a (0,3) + b (\tfrac{3\sqrt{3}}{2},\tfrac32) \right\| = 3 \sqrt{a^2+ab+b^2}.
    \]
    As such, the length of $\tilde\gamma$ lies in the set
    \[
    \{\tfrac 32\sqrt{a^2+ab+b^2} : a,b\in \ZZ\}. 
    \]
    Note that $m\frac 32$ is in this set (take $a=m,b=0$) so this covers the (2.b'') case as well. This completes the proof. 
\end{proof}

\section{The first width of a convex polygon}\label{sec:w1} In this section we prove Theorem \ref{theo:poly-w1}. Let $P\subset \RR^2$ be a convex polygon. For $v \in S^1$ we define
\[
\Phi : [-L,L] \to \cZ_1(P,\partial P), \qquad t\mapsto \partial\{p \in P : p \cdot v \leq t\}.
\]
Fix $L$ sufficiently large so that
\[
\{p \in P : p \cdot v \leq -L\} = \emptyset, \qquad \{p \in P : p \cdot v \leq L\} = P.
\]
It is clear that $\Phi$ is a $1$-sweepout with
\[
\max_{t \in [-L,L]} \MM(\Phi(t)) \leq \max_{x\in K} x\cdot v - \min_{x\in K} x\cdot v 
\]
Taking the infimum over all such $v$ we get $\omega_1(P) \leq W(P)$. 

For the converse inequality, we consider any smoothing $P_s \subset P$ so that $P_s$ is a strictly convex region with smooth boundary and $P_s$ limits to $P$ as $s\to 0$ (e.g.\ in the Hausdorff sense). By \cite{DonatoMontezuma}, $\omega_1(P_s)$ is equal to the length of some free boundary line segment $\eta$, i.e. with $\partial \eta \subset \partial P_s$ and $\eta$ meets $\partial P_s$ orthogonally. Since $P_s$ is convex, if we let $v$ denote a unit vector in the direction of $\eta$ we thus see that
\[
\omega_1(P_s) \geq \max_{x\in P_s } x\cdot v - \min_{x\in P_s} x\cdot v \geq W(P_s) \to W(P)
\]
as $s\to 0$. Since $P_s\subset P$ we have $\omega_1(P_s) \leq \omega_1(P)$. Putting this together we have proven Theorem \ref{theo:poly-w1}.

\section{Low $p$-widths of the equilateral triangle}\label{sec:triangle}

Let $T$ be the equilateral triangle inscribed in the unit circle. Note that $W(T) = \frac 32$. Thus Theorem \ref{theo:poly-w1} gives $\omega_1(T) = \frac 32$. 

\subsection{$w_2(T)$} \label{subsec:2-width-T} In this section we construct a $2$-sweepout $\Phi :\RR P^2\to \cZ_1(T,\partial T)$ with $\max_{x\in \RR P^2} \MM(\Phi(x)) = \frac 32$. Since $\frac32=\omega_1(T)\leq \omega_2(T)$ this will prove that $\omega_2(T) = \frac 32$. 

We first define 
\[
\phi:\partial T\times \partial T\to \cZ_{1}(T,\partial T)
\]
as follows. If $p_i \in \partial T\setminus V$ we write $l_i$ for the line through $p_i$ perpendicular to $\partial T$ at $p_i$. Then:
\begin{itemize}
    \item If $p_1,p_2$ are on the same edge of $T$, set $\phi(p_1,p_2)=\emptyset$. This includes the case where both points are vertices of $\partial T$.
    \item If $p_1 \in V$ and $p_2$ is not, set $\phi(p_1,p_2)=l_2\cap T$, and vice versa.
    \item Otherwise, let $l_1,l_2$ meet at $p$, and set $\phi(p_1,p_2) = (\overline{p_1p}\cup \overline{p_2p}) \cap T$. 
\end{itemize} 
See Figure \ref{fig:2-width-T}. 

\begin{figure}[h]
    \centering
    \includegraphics[width=9cm]{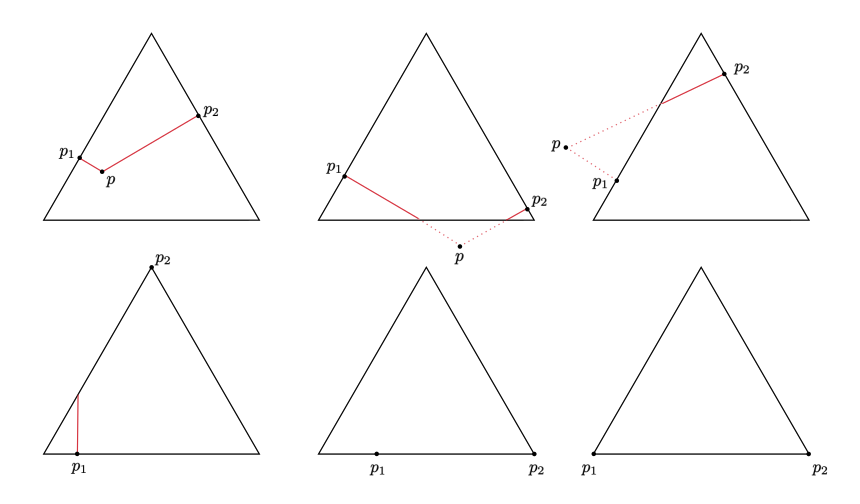}\label{figure:triangle-phi-map-construction}
    \caption{The definition of the map $\phi : \partial T\times \partial T \to \cZ_1(T,\partial T)$ used to construct an optimal $2$-sweepout of the equilateral triangle $T$.}
    \label{fig:2-width-T}
\end{figure}

\begin{lemm}\label{lemm:Phi-2-swp-upper-bd}
    $\max_{(p_1,p_2)\in\partial T\times \partial T} \MM(\phi(p_1,p_2)) = \frac 32$
\end{lemm}
\begin{proof}
    We can assume that $p_1,p_2$ do not lie on the same edge. If $p_1$ (or equivalently $p_2$) is a vertex then $\phi(p_1,p_2) = l_2\cap T$ has length $\leq \frac 32$ (with equality when $p_2$ is the midpoint of the opposite edge from $p_1$). Thus it suffices to assume that neither of $p_1,p_2$ is a vertex. 

    We first consider the case where $l_1\cap l_2 = p \in T$. Let $L$ be the line through $p$ that's parallel to the edge of $T$ not containing $p_1,p_2$. Let $t_1,t_2$ denote the points in $L\cap \partial T$ (with $t_i,p_i$ on the same edge for $i=1,2$). See Figure \ref{fig:upper-bd-phi-1}. 

    We have
    \[
    |\overline{p_1p}| + |\overline{p_2p}| = \sin \left(\tfrac{\pi}{3}\right) \left(|\overline{t_1p}| + |\overline{t_2p}|\right) = \tfrac{\sqrt{3}}{2}|\overline{t_1t_2}| \leq \tfrac 32. 
    \]
    The first equality follows by considering the $t_ip_ip$ triangles and the final inequality follows from the fact that $|\overline{t_1t_2}| \leq \sqrt{3}$ (the side length of $T$).  
    \begin{figure}[h]
        \centering
        \includegraphics[width=4cm]{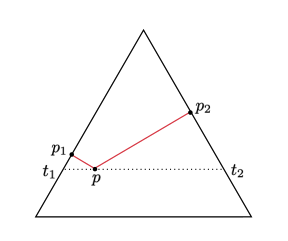}
        \caption{Deriving an upper bound for the length of $\phi$ when $l_1\cap l_2 = p \in T$.}
        \label{fig:upper-bd-phi-1}
    \end{figure}

    Next, we consider the case where $l_1\cap l_2 = p \not \in T$. Suppose that $p$ lies on the same side of $T$ as $p_1$ (the case where $p$ lies on the same side of $T$ as $p_2$ is similar). See Figure \ref{fig:upper-bd-phi-2}. In this case, $\phi(p_1,p_2) = l_2\cap T$, so we have $\MM(\phi(p_1,p_2))\leq \frac 32$. 
    \begin{figure}[h]
        \centering
        \includegraphics[width=4cm]{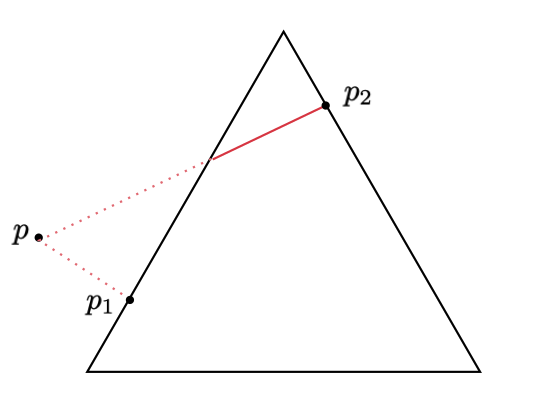}
        \caption{Deriving an upper bound for the length of $\phi$ when $l_1\cap l_2 = p$ lies outside of $T$ but on the same side of $T$ as $p_1$.}
        \label{fig:upper-bd-phi-2}
    \end{figure}

    Finally, we assume that $l_1\cap l_2 = p \not \in T$ is on the opposite side of the edge $\Gamma$ of $T$ not containing $p_1,p_2$. See Figure \ref{fig:upper-bd-phi-3}.
    \begin{figure}[h]
        \centering
        \includegraphics[width=4cm]{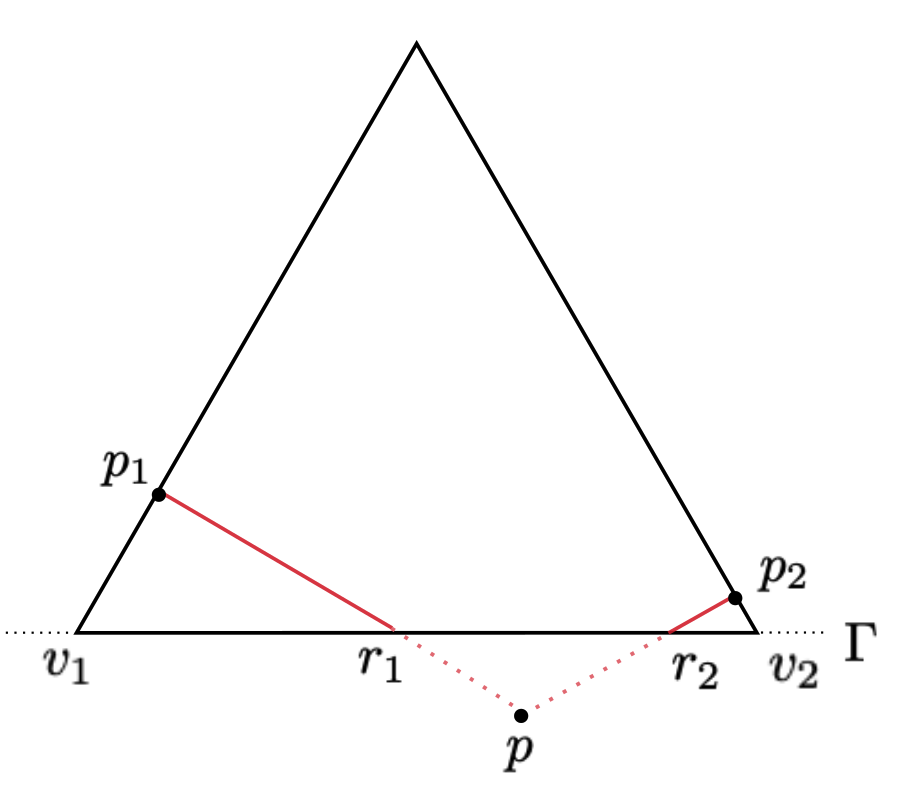}
        \caption{Deriving an upper bound for the length of $\phi$ when $l_1\cap l_2 = p$ lies outside of $T$ on the opposite side of $T$ as $p_1,p_2$.}
        \label{fig:upper-bd-phi-3}
    \end{figure}
    Let $r_i = l_i \cap \Gamma$ and $v_1,v_2$ denote the vertices of $T$ on $\Gamma$. We have
    \[
    |\overline{p_1r_1}| + |\overline{p_2r_2}| = \sin\left(\tfrac \pi 3\right) \left( |\overline{v_1r_1}| + |\overline{v_2r_2}|\right) \leq \tfrac{\sqrt{3}}{2}|\overline{v_1v_2}| = \tfrac32.
    \]
    This completes the proof.
\end{proof}

\begin{lemm}
    $\phi$ is continuous.
\end{lemm}
\begin{proof}
    It suffices to check continuity at $(p_1,p_2)$ where one of the $p_i$ (without loss of generality $p_1$) is a vertex. If $p_2=p_1$ or $p_2$ lies on one of the two edges containing $p_1$ then for $(p_1',p_2')$ close to $(p_1,p_2)$ it's easy to see that $\phi(p_1',p_2')$ bounds a region of small area and is thus close to $\phi(p_1,p_2)=\emptyset$ in the flat norm. If $p_2$ lies on the opposite edge then $\phi(p_1,p_2)=l_2\cap T$ and a similar argument applies. 
\end{proof}

Note that $\phi(p_2,p_1) = \phi(p_1,p_2)$ and thus $\phi$ descends to a continuous map $\Sym^2(\partial T) \to \cZ_1(T,\partial T)$. Let $W\subset \Sym^2(\partial T)$ be the set of $[(p_1,p_2)]$ so that $p_1,p_2$ lie on a common edge. Since $\phi|_{W}=\emptyset$ we actually have a map $\Sym^2(\partial T)/W \to \cZ_1(T,\partial T)$. 

We now define $\psi : \RR P^2 \to \Sym^2(\partial T)/W$ as follows. Suppose that $a x + by + c = 0$ intersects $\partial T$ in precisely two points $(p_1,p_2)$. Then we set $\psi([a:b:c]) = [(p_1,p_2)]$. Otherwise (i.e.\ the intersection of the line and $\partial T$ is a full edge, only a vertex, or empty) we set $\phi([a:b:c]) = [W]$. 
\begin{lemm}
    $\psi$ is continuous. 
\end{lemm}
\begin{proof}
    The set $[a:b:c]\in \RR P^2$ so that $\{ax+by+c=0\}$ intersects $\partial T$ in precisely two points or not at all is open and in this set continuity is clear. Suppose that $[a:b:c] \in \RR P^2$ has $\{ax+by+c=0\} \cap \partial T = v \in V$ (so $\psi([a:b:c]) = [W]$, nearby points $[a':b':c']$ either map to lines that avoid $T$ or else intersect $T$ in two points very close to $v$. A similar consideration when $\{ax+by+c=0\} \cap \partial T$ is an edge of $T$ proves continuity. 
\end{proof}

\begin{lemm}\label{lemm:Phi-2-swp-T-is-swp}
    $\Psi := \psi\circ\phi : \RR P^2 \to \cZ_1(T,\partial T)$ is a $2$-sweepout with no concentration of mass.
\end{lemm}
\begin{proof}
    We have that $\Psi$ is continuous and has no concentration of mass (since $\Psi([a:b:c])$ consists of at most two line segments in $T$). Let $\bar \lambda$ generate $H^*(\cZ_1(T,\partial T);\ZZ_2)$. Since $H^*(\RR P^2;\ZZ_2) = \ZZ_2[\alpha]/(\alpha^3)$ it suffices to prove that $\Phi^*(\bar \lambda) \neq 0$ and in fact, by Lemma \ref{lemm:check-sweepout} it suffices to prove that for $\gamma \subset \RR P^2$ generating $\pi_1(\RR P^2)$, $\Phi \circ\gamma$ is a $1$-sweepout.

    Fix $v \in V$ and Let $\gamma\subset \RR P^2$ denote the loop of lines in $\RR^2$ through $v$. Write $E$ for the edge of $T$ across from $V$. If $\gamma(t)$ represents a line that intersects $E$ (in the interior) then $\Psi(t)$ will be the line segment in $T$ through $\gamma(t)\cap E$ (meeting $E$ perpendicularly). Thus, $\Phi \circ \gamma$ is the sweepout of $T$ by lines parallel to $E$ and is thus a $1$-sweepout. 
\end{proof}

Putting these lemmas together we have proven $\omega_2(T)\leq \frac32$ which implies $\omega_2(T) = \frac32$ as explained previously. 

\subsection{$w_3(T)$}

\begin{lemm}
    $\omega_3(T) \geq \frac{3\sqrt{3}}{2} $
\end{lemm}
\begin{proof}
We can divide $T$ into four equilateral triangles congruent to $\frac 12 T$. Lemma \ref{lemm:LS} and $\omega_1(T) = \frac 32$ gives
\begin{equation}\label{eq:w3-T-LS-lb}
\omega_3(T) \geq 3 \omega_1(\tfrac 1 2 T) = \tfrac 9 4.
\end{equation}
On the other hand, Proposition \ref{prop:triangle-billiards} implies that 
\[
\omega_3(T) = \sum_{j=1}^N \length(\gamma_j)
\]
for $T$-billiards $\gamma_j$. Lemma \ref{lemm:len-T-Bil} gives
\[
\length(\gamma_j) \in \{\tfrac 32 \sqrt{a^2+ab+b^2} : a,b \in \ZZ\}
\]
Note that the only elements of this set with length $\leq \frac{3\sqrt{3}}{2}$ are $\frac 32$ (with $a=1,b=0$ or $a=0,b=1$) and $\frac{3\sqrt{3}}{2}$ (with $a=b=1$). Thus, we find that if $\omega_3(T) < \frac{3\sqrt{3}}{2}$ then it is $\frac 32$. But $\frac 32 < \frac 9 4$, contradicting \eqref{eq:w3-T-LS-lb}. This completes the proof. 
\end{proof}

\begin{lemm}\label{lemm:tet-intersection-length}
    If $Q$ is a regular tetrahedron of side length $\ell$, let $\cP$ denote the set of affine planes in $\RR^3$ that do not contain one of the sides of $Q$. Then
    \[
    \sup_{\Pi \in \cP} \length(\Pi \cap \cQ) = 3\ell. 
    \]
\end{lemm}
\begin{proof}
    Consider $\Pi \in \cP$ and let $\Pi_t$ be the translation of $\Pi$ in the normal direction with speed $t$. Note that $t\mapsto \length(\Pi_t \cap Q)$ is linear as long as $\Pi_t$ does not intersect any vertices of $Q$. Indeed, for such an interval, $\Pi_t \cap Q$ will be a polygon whose vertices move linearly with $t$ in $\Pi_t$ (since they are the intersection of $\Pi_t$ with a line in $\RR^3$). Thus, we can increase/decrease $t$ to increase $\length(\Pi_t \cap Q)$ until either (a) $\Pi_T$ is a face of $Q$ (in which case we find that $\lim_{t\to T}\length(\Pi_t\cap Q) = 3\ell$), or (b) $\Pi_T$ contains one or two vertices of $Q$. In the case (b), $\Pi_T\cap Q$ must consist of three line segments each contained in a face. The maximal length of such a segment is $\ell$. This completes the proof. 
    \end{proof}

\begin{prop}\label{prop:third-width-triangle}
$\omega_3(T)=\frac{3\sqrt{3}}{2}$.
\end{prop}
\begin{proof}
    For $\eps>0$, consider $X_\eps\subset \RR^3$ the compact set bounded by the regular tetrahedron of side length $\frac{\sqrt{3}}{2}+2\eps$. Let $M_\eps' = \partial U_\eps(X_\eps)$ be the boundary of the tubular neighborhood of $X_\eps$. Observe that the induced metric on $M_\eps'$ has zero Gaussian curvature except on $M_\eps'\cap B_\eps(V)$. In particular, since $T$ can be thought of as a net for the tetrahedron of side length $\frac{\sqrt{3}}{2}$, we can find a region $T_\eps \subset M'_\eps$ that is isometric to $T$. 
    
    Thus, we have that $\omega_3(T) \leq \omega_3(M_\eps')$ by Lemma \ref{lemm:LS}. Let $M_\eps$ denote a strictly convex approximation of $M_\eps'$ so that still $M_\eps$ converges to $T$ as $\eps \to 0$. Thus it suffices to prove that
    \[
    \lim_{\eps\to 0} \omega_3(M_\eps) \leq \frac{3\sqrt{3}}{2}.
    \]
    Consider $\Phi : \RR P^3 \to \cZ_1(M_\eps)$ defined by $\Phi([a:b:c:d]) = \partial(M_\eps\cap \{ax+by+cz+d<0\})$. It's easy to see this is a $3$-sweepout with no concentration of mass. 
    
    Let $Q_\eps$ denote the concentric tetrahedron of side length $\frac{\sqrt{3}}{2} + C\eps$ for $C$ sufficiently large so that $M_\eps$ is strictly contained inside of $Q_\eps$. For any plane $\Pi=\{ax+by+cz+d=0\}$ so that $\Pi \cap M_\eps \neq \emptyset$ we have that either $\Pi$ is tangent to $M_\eps$ or else intersects $M_\eps$ transversely. Moreover, in either case, $\Pi$ must intersect each plane defining $Q_\eps$ transversely, if at all ($\Pi$ cannot be equal to any of the planes since this would violate $\Pi \cap M_\eps \neq \emptyset$). As such, if $\Pi$ intersects $M_\eps$ transversely then we have that $\Pi \cap M_\eps$ is a convex curve in $\RR^2$ contained inside (as curves in $\Pi$) of the convex curve $\Pi \cap Q_\eps$. Thus, the Crofton formula implies that 
    \[
    \length(\Pi\cap M_\eps) \leq \length(\Pi\cap Q_\eps) \leq \frac{3\sqrt{3}}{2} + 3 C\eps,
    \]
    where we used Lemma \ref{lemm:tet-intersection-length} in the last step. Sending $\eps\to 0$ completes the proof. 
\end{proof}

\subsection{$w_4(T)$} Consider the $2$-sweepout $\Phi : \RR P^2\to \cZ_1(T,\partial T)$ defined in Section \ref{subsec:2-width-T}. Let $\hat \Phi : \Sym^2(\RR P^2) \to \cZ_1(T,\partial T)$ be $\hat \Phi([x,y]) = \Phi(x) + \Phi(y)$. By \cite[\S 4]{DL} (cf.\ \cite{Arnold}), $\Sym^2(\RR P^2)$ is homotopy equivalent\footnote{Actually with the natural smooth structure on $\Sym^n(\RR P^2)$ we see that $\Sym^2(\RR P^2)$ is diffeomorphic to $\RR P^{2n}$.} to $\RR P^4$. From this (and Lemma \ref{lemm:Phi-2-swp-T-is-swp}) we see that $\hat \Phi$ is a $4$-sweepout with no concentration of mass. By Lemma \ref{lemm:Phi-2-swp-upper-bd} we have $\sup_{(x,y) \in \Sym^2(\RR P^2)} \MM(\Phi(x,y)) \leq 3$. Thus we find that
\[
\omega_4(T) \leq 3. 
\]
As above, we can divide $T$ into four equilateral triangles congruent to $\frac 12 T$. Lemma \ref{lemm:LS} and $\omega_1(T) = \frac 32$ gives
\[
\omega_4(T) \geq 4 \omega_1(\tfrac 1 2 T) = 3.
\]
Thus, $\omega_4(T) = 3$ as claimed. 

\section{Low $p$-widths of the square}\label{sec:square}
Let $S$ be the square (with side length $\sqrt{2}$) inscribed in the unit circle. $W(S)=\sqrt{2}$, so Theorem \ref{theo:poly-w1} gives $\omega_1(S)=\sqrt{2}$.

\subsection{$\omega_2(S)$.} 
\begin{prop}
    $\omega_2(S)=2$.
\end{prop}
\begin{proof}
Divide $S$ along one of its diagonal into two congruent right isoceles triangles $T_1,T_2$. Note that $\omega_1(T_1)=\omega_1(T_2)=W(T_1)=W(T_2)=1$. Thus, Lemma \ref{lemm:LS} gives
$$\omega_2(S)\geq \omega_1(T_1)+\omega_2(T_2)=2$$
On the other hand, a the $2$-sweepout $[a:b:c]\in \RR P^2 \mapsto \{ax+by=c\}\cap S$ gives $\omega_2(S)\leq 2$.
\end{proof}

\subsection{$\omega_3(S)$.} 

\begin{prop}
  $\omega_3(S)=2\sqrt{2}$.
\end{prop}
\begin{proof}
Dividing $S$ into 4 congruent squares each with half the side length of $S$, we get from Lemma \ref{lemm:LS} and Theorem \ref{theo:poly-w1} that
\[
\omega_3(S)\geq 3\omega_1\left(\frac{1}{2}S\right)=\frac{3\sqrt{2}}{2}
\]
By an argument similar to Lemma \ref{lemm:triangle-billiard-reflect}, widths of $S$ are achieved by a set of line segments each connecting two points in $(\sqrt{2}\ZZ)^2$. So, $\omega_3(S)=\sum_{j=1}^{N'} \length(\gamma_j)$ with $\length(\gamma_j)\in \{\sqrt{2}\sqrt{a^2+b^2}: a,b\in\ZZ\}$. Since $\omega_3(S)\geq \frac{3\sqrt{2}}{2}$ (so in particular it is greater than $2$), we must have $\omega_3(S)\geq 2\sqrt{2}$.

Next we consider $\Phi : \RR P^3  \to \cZ_1(S,\partial S)$ defined by
\begin{align*}
 [a:b:c:d]  \mapsto \partial (S\cap \{axy+bx+cy+d<0\})
\end{align*}
It's clear that $\Phi$ is a $3$-sweepout with no concentration of mass. As such, it remains to bound $\max_{\RR P^3}\MM(\Phi(\cdot))$. Note that $\Phi([a:b:c:d])$ is the intersection of a hyperbola $H$ (with vertical/horizontal axes) with $S$. Write $r = \frac{\sqrt{2}}{2}$. If $\MM(\Phi([a:b:c:d])) \neq 0$ the corresponding hyperbola intersects $S$ in one or two arcs, each of which is a graph of an increasing/decreasing function. Since $H$ projects injectively to the coordinate axes, the assertion follows from Lemma \ref{lemm:graph-increasing-func} below.
\end{proof}

\begin{lemm}\label{lemm:graph-increasing-func}
    Suppose that $f: [a,b] \to \RR$ is a $C^1$-function with $f'\geq0$. The length of $[a,b] \ni t \mapsto (t,f(t))$ is bounded by $( b-a)+ (f(b)-f(a))$. 
\end{lemm}
\begin{proof}
    The length satisfies
    \begin{align*}
        L & = \int_a^b \sqrt{1+f'(t)^2} \, dt \\
        & \leq \int_a^b 1+|f'(t)| \, dt\\
        & = (b-a) + (f(b)-f(a)). 
    \end{align*}
    This completes the proof. 
\end{proof}

\appendix

\section{Caccioppoli sets}\label{app:caccioppoli-sets}

Fix $K\subset \RR^2$ a compact convex set with Lipschitz boundary. For $\Omega\subset K$ measurable we define the \emph{relative perimeter} of $\Omega$ by
\[
P(\Omega) : = \sup\left\{\int_\Omega \Div X : X \in C^1(K;\RR^2), |X|\leq 1, \supp X \Subset \mathring K \right\}
\]
for $\mathring K$ the interior of $K$. When $\partial \Omega$ is a smooth curve this can easily be seen to agree with the length of $\partial\Omega$ (as a curve in $\mathring K$). We call $\Omega$ a \emph{Caccioppoli set} (set of finite perimeter) in $K$ if $P(\Omega) < \infty$.  An overview of Caccioppoli sets may be found in \cite{Maggi:FP}. We let $\cC(K,\partial K)$ denote the space of Caccioppoli sets (identifying sets that agree up to measure zero) equipped with the $L^1$-topology, i.e.\ $d_\cC(\Omega,\tilde\Omega) = \| \chi_{\Omega} - \chi_{\tilde\Omega}\|_{L^1(K)}$. 

Key facts underlying the Gromov--Guth theory of $p$-widths are:
\begin{enumerate}
\item $\cC(K,\partial K)$ is contractible. To see this, we let $\Phi_t : \cC(K,\partial K)\to \cC(K,\partial K)$, $\Omega \mapsto \Omega\cap\{x \leq t\}$. This is well-defined since the intersection of two Caccioppoli sets is again Caccioppoli \cite[Lemma 12.22]{Maggi:FP}. Since $\Phi_t = \Id$ for $t\gg0$ and $\Phi_t(\Omega)= \emptyset$ for $t\ll 0$, the assertion follows. 
\item $\Omega\mapsto \Omega^c$ is a fixed-point free $\ZZ_2$ action on $\cC(K,\partial K)$ with $P(\Omega) = P(\Omega^c)$. 
\end{enumerate}
As such, we can define $\cZ_1(K,\partial K) : = \cC(K,\partial K)/(\Omega\sim \Omega^c)$. For $\Omega \in \cC(K,\partial K)$ we'll use the notation $\partial\Omega$ for the equivalence class in $\cZ_1(K,\partial K)$. We write $\MM(\partial \Omega) = P(\Omega)$ (well-defined by (2) above).  

Note that the quotient topology can be metrized via the \emph{flat norm} as 
\[
\cF(\partial\Omega,\partial \tilde\Omega) = \min\{d_\cC(\Omega,\tilde\Omega),d_\cC(\Omega,\tilde\Omega^c)\}.
\]
One may give a geometric interpertation of $\cZ_1(K,\partial K)$ and the boundary operator $\partial$ using the notion of rectifiable mod $2$ flat chains (cf.\ \cite{Pitts,Simon:GMT,Guth,MarquesNeves:app}).

\bibliographystyle{alpha}
\bibliography{bib}

\begin{thebibliography}{ABCS19}

\bibitem[ABCS19]{ABCS}
Lucas Ambrozio, Reto Buzano, Alessandro Carlotto, and Ben Sharp.
\newblock Bubbling analysis and geometric convergence results for free boundary minimal surfaces.
\newblock {\em J. \'Ec. polytech. Math.}, 6:621--664, 2019.

\bibitem[Aie19]{Aiex:ellipsoids}
Nicolau~Sarquis Aiex.
\newblock The width of ellipsoids.
\newblock {\em Comm. Anal. Geom.}, 27(2):251--285, 2019.

\bibitem[AK82]{AK:elliptic}
A.~Azzam and E.~Kreyszig.
\newblock On solutions of elliptic equations satisfying mixed boundary conditions.
\newblock {\em SIAM J. Math. Anal.}, 13(2):254--262, 1982.

\bibitem[Alm62]{Almgren:htpy}
Frederick~Justin Almgren, Jr.
\newblock The homotopy groups of the integral cycle groups.
\newblock {\em Topology}, 1:257--299, 1962.

\bibitem[Arn96]{Arnold}
V.~Arnold.
\newblock {Topological content of the Maxwell theorem on multiple representation of spherical functions}.
\newblock {\em Topological Methods in Nonlinear Analysis}, 7(2):205 -- 217, 1996.

\bibitem[BdL23]{BL:RPn}
M\'arcio Batista and Anderson de~Lima.
\newblock The first and second widths of the real projective space.
\newblock {\em Proc. Amer. Math. Soc.}, 151(9):3985--3997, 2023.

\bibitem[BL22]{BL:RP3}
M\'arcio Batista and Anderson Lima.
\newblock Min-max widths of the real projective 3-space.
\newblock {\em Trans. Amer. Math. Soc.}, 375(7):5239--5258, 2022.

\bibitem[BL23]{BL:lens}
M\'arcio Batista and Anderson Lima.
\newblock A short note about 1-width of lens spaces.
\newblock {\em Bull. Math. Sci.}, 13(1):Paper No. 2250005, 10, 2023.

\bibitem[Chu]{Chu:database}
Adrian Chun-Pong Chu.
\newblock The p-widths database.
\newblock {\em \url{https://sites.google.com/view/adrianchu/the-p-widths-database}}.

\bibitem[Chu23]{Chu:ball}
Adrian Chun-Pong Chu.
\newblock A free boundary minimal surface via a 6-sweepout.
\newblock {\em J. Geom. Anal.}, 33(7):Paper No. 230, 31, 2023.

\bibitem[CL23]{ChuLi}
Adrian Chun-Pong Chu and Yangyang Li.
\newblock A strong multiplicity one theorem in min-max theory.
\newblock {\em \url{https://arxiv.org/abs/2309.07741}}, 2023.

\bibitem[CM20]{ChodoshMantoulidis:3d}
Otis Chodosh and Christos Mantoulidis.
\newblock Minimal surfaces and the {A}llen-{C}ahn equation on 3-manifolds: index, multiplicity, and curvature estimates.
\newblock {\em Ann. of Math. (2)}, 191(1):213--328, 2020.

\bibitem[CM23]{ChodoshMantoulidis:2d}
Otis Chodosh and Christos Mantoulidis.
\newblock The {$p$}-widths of a surface.
\newblock {\em Publ. Math. Inst. Hautes \'Etudes Sci.}, 137:245--342, 2023.

\bibitem[Dey22]{Dey:comp}
Akashdeep Dey.
\newblock A comparison of the {A}lmgren-{P}itts and the {A}llen-{C}ahn min-max theory.
\newblock {\em Geom. Funct. Anal.}, 32(5):980--1040, 2022.

\bibitem[DL71]{DL}
J.~L. Dupont and G.~Lusztig.
\newblock On manifolds satisfying {$w\sb{1}{}\sp{2}=0$}.
\newblock {\em Topology}, 10:81--92, 1971.

\bibitem[DM24]{DonatoMontezuma}
Sidney Donato and Rafael Montezuma.
\newblock The first width of non-negatively curved surfaces with convex boundary.
\newblock {\em J. Geom. Anal.}, 34(2):Paper No. 60, 28, 2024.

\bibitem[Don22]{Donato}
Sidney Donato.
\newblock The first {$p$}-widths of the unit disk.
\newblock {\em J. Geom. Anal.}, 32(6):Paper No. 177, 38, 2022.

\bibitem[GG18]{GasparGuaraco1}
Pedro Gaspar and Marco A.~M. Guaraco.
\newblock The {A}llen-{C}ahn equation on closed manifolds.
\newblock {\em Calc. Var. Partial Differential Equations}, 57(4):Paper No. 101, 42, 2018.

\bibitem[Gro03]{Gromov:waist}
M.~Gromov.
\newblock Isoperimetry of waists and concentration of maps.
\newblock {\em Geom. Funct. Anal.}, 13(1):178--215, 2003.

\bibitem[Gro09]{Gromov:expanders}
Mikhail Gromov.
\newblock Singularities, expanders and topology of maps. {I}. {H}omology versus volume in the spaces of cycles.
\newblock {\em Geom. Funct. Anal.}, 19(3):743--841, 2009.

\bibitem[Gro15]{Gromov:spectra}
M.~Gromov.
\newblock Morse spectra, homology measures, spaces of cycles and parametric packing problem.
\newblock {\em preprint}, 2015.

\bibitem[Gua18]{Guaraco:1param}
Marco A.~M. Guaraco.
\newblock Min-max for phase transitions and the existence of embedded minimal hypersurfaces.
\newblock {\em J. Differential Geom.}, 108(1):91--133, 2018.

\bibitem[Gut09]{Guth}
Larry Guth.
\newblock Minimax problems related to cup powers and {S}teenrod squares.
\newblock {\em Geom. Funct. Anal.}, 18(6):1917--1987, 2009.

\bibitem[HT00]{HT}
John~E. Hutchinson and Yoshihiro Tonegawa.
\newblock Convergence of phase interfaces in the van der {W}aals-{C}ahn-{H}illiard theory.
\newblock {\em Calc. Var. Partial Differential Equations}, 10(1):49--84, 2000.

\bibitem[IMN18]{IrieMarquesNeves}
Kei Irie, Fernando~C. Marques, and Andr\'e Neves.
\newblock Density of minimal hypersurfaces for generic metrics.
\newblock {\em Ann. of Math. (2)}, 187(3):963--972, 2018.

\bibitem[Lan23]{Lange:billiard}
Christian Lange.
\newblock On continuous billiard and quasigeodesic flows characterizing alcoves and isosceles tetrahedra.
\newblock {\em J. Lond. Math. Soc. (2)}, 107(5):1754--1779, 2023.

\bibitem[LMN18]{LiokumovichMarquesNeves}
Yevgeny Liokumovich, Fernando~C. Marques, and Andr\'e Neves.
\newblock Weyl law for the volume spectrum.
\newblock {\em Ann. of Math. (2)}, 187(3):933--961, 2018.

\bibitem[LPS24]{LPS}
Martin Man-chun Li, Davide Parise, and Lorenzo Sarnataro.
\newblock Boundary behavior of limit-interfaces for the allen--cahn equation on riemannian manifolds with neumann boundary condition.
\newblock {\em Archive for Rational Mechanics and Analysis}, 248(6):124, 2024.

\bibitem[Lun19]{RL}
Alejandra~Ram\'irez Luna.
\newblock Compact minimal hypersurfaces of index one and the width of real projective spaces.
\newblock {\em \url{https://arxiv.org/abs/1902.08221}}, 2019.

\bibitem[LW22]{LW}
Yong Liu and Juncheng Wei.
\newblock Classification of finite {M}orse index solutions to the elliptic sine-{G}ordon equation in the plane.
\newblock {\em Rev. Mat. Iberoam.}, 38(2):355--432, 2022.

\bibitem[Mag12]{Maggi:FP}
Francesco Maggi.
\newblock {\em Sets of finite perimeter and geometric variational problems}, volume 135 of {\em Cambridge Studies in Advanced Mathematics}.
\newblock Cambridge University Press, Cambridge, 2012.
\newblock An introduction to geometric measure theory.

\bibitem[Mat15]{mathoverflow:laplace.spectrum}
MathOverflow.
\newblock {Explicit eigenvalues of the Laplacian}.
\newblock \url{https://mathoverflow.net/questions/219109/explicit-eigenvalues-of-the-laplacian}, 2015.
\newblock [Accessed 19-July-2021].

\bibitem[MK25]{MK:Rp2}
Jared Marx-Kuo.
\newblock The p-widhts of $rp^2$.
\newblock {\em \url{https://arxiv.org/abs/2501.13311}}, 2025.

\bibitem[MN17]{MarquesNeves:posRic}
Fernando~C. Marques and Andr\'e Neves.
\newblock Existence of infinitely many minimal hypersurfaces in positive {R}icci curvature.
\newblock {\em Invent. Math.}, 209(2):577--616, 2017.

\bibitem[MN20]{MarquesNeves:app}
Fernando~C. Marques and Andr\'e Neves.
\newblock Applications of min-max methods to geometry.
\newblock In {\em Geometric analysis}, volume 2263 of {\em Lecture Notes in Math.}, pages 41--77. Springer, Cham, [2020] \copyright 2020.

\bibitem[MN21]{MarquesNeves:uper-semi-index}
Fernando~C. Marques and Andr\'{e} Neves.
\newblock Morse index of multiplicity one min-max minimal hypersurfaces.
\newblock {\em Adv. Math.}, 378:107527, 58, 2021.

\bibitem[Mod85]{Modica}
Luciano Modica.
\newblock A gradient bound and a {L}iouville theorem for nonlinear {P}oisson equations.
\newblock {\em Comm. Pure Appl. Math.}, 38(5):679--684, 1985.

\bibitem[Nur16]{Nurser}
Charles Nurser.
\newblock {\em Low min-max widths of the round three-sphere}.
\newblock PhD thesis, Imperial College London, 180 Queen's Gate, London SW7 2BZ, October 2016.

\bibitem[Pit81]{Pitts}
Jon~T. Pitts.
\newblock {\em Existence and regularity of minimal surfaces on {R}iemannian manifolds}, volume~27 of {\em Mathematical Notes}.
\newblock Princeton University Press, Princeton, N.J.; University of Tokyo Press, Tokyo, 1981.

\bibitem[Sim83]{Simon:GMT}
Leon Simon.
\newblock {\em Lectures on geometric measure theory}, volume~3 of {\em Proceedings of the Centre for Mathematical Analysis, Australian National University}.
\newblock Australian National University, Centre for Mathematical Analysis, Canberra, 1983.

\bibitem[Son23]{Song:full-yau}
Antoine Song.
\newblock Existence of infinitely many minimal hypersurfaces in closed manifolds.
\newblock {\em Ann. of Math. (2)}, 197(3):859--895, 2023.

\bibitem[WW19a]{WangWei}
Kelei Wang and Juncheng Wei.
\newblock Finite {M}orse index implies finite ends.
\newblock {\em Comm. Pure Appl. Math.}, 72(5):1044--1119, 2019.

\bibitem[WW19b]{WangWei2}
Kelei Wang and Juncheng Wei.
\newblock Second order estimate on transition layers.
\newblock {\em Adv. Math.}, 358:106856, 85, 2019.

\bibitem[Zho20]{Zhou:multiplicity-one}
Xin Zhou.
\newblock On the multiplicity one conjecture in min-max theory.
\newblock {\em Ann. of Math. (2)}, 192(3):767--820, 2020.

\bibitem[ZZ19]{ZZ:cmc}
Xin Zhou and Jonathan~J. Zhu.
\newblock Min-max theory for constant mean curvature hypersurfaces.
\newblock {\em Invent. Math.}, 218(2):441--490, 2019.

\end{thebibliography}
\end{document}